\documentclass[reqno,12pt]{amsart}

\usepackage[usenames, dvipsnames]{color}

\usepackage{latexsym}
\usepackage{amssymb}
 
\usepackage{amsfonts}
\usepackage{amsmath,amsthm,amscd}
\usepackage{graphicx}
\usepackage{subfigure}
\usepackage[sorted-cites,abbrev]{amsrefs}
\newcommand{\black}{\color{black}}
\newcommand{\bu}{{\mathbf u}}
\newcommand{\bv}{{\mathbf v}}

\renewcommand{\phi}{\varphi}
\newcommand{\pa}{\partial}

\newcommand{\uN}{\frac{u}{\sqrt N}}
\newcommand{\vN}{\frac{v}{\sqrt N}}

\renewcommand{\Re}{{\operatorname{Re}\,}}

\renewcommand{\epsilon}{\varepsilon}
\renewcommand{\d}{\partial}

\newcommand{\e}{{\mathbf e}}

\newcommand{\szego}{Szeg\H{o} }

\newcommand{\inv}{^{-1}}
\newcommand{\kahler}{K\"ahler }
\newcommand{\sqrtn}{\sqrt{N}}
\newcommand{\wt}{\widetilde}
\newcommand{\wh}{\widehat}
\newcommand{\PP}{{\mathbb P}}
\newcommand{\Q}{{\mathbb Q}}

\newcommand{\R}{{\mathbb R}}
\newcommand{\C}{{\mathbb C}}
\newcommand{\Z}{{\mathbb Z}}

\newcommand{\CP}{\C\PP}
\renewcommand{\dbar}{\bar\partial}
\newcommand{\ddbar}{\partial\dbar}

\newcommand{\half}{{\textstyle \frac 12}}

\newcommand{\ccal}{\mathcal{C}}

\newcommand{\fcal}{\mathcal{F}}

\newcommand{\hcal}{\mathcal{H}}

\newcommand{\lcal}{\mathcal{L}}

\newcommand{\pcal}{\mathcal{P}}

\newcommand{\rcal}{\mathcal{R}}

\newcommand{\al}{\alpha}
\newcommand{\be}{\beta}

\newcommand{\om}{\omega}

\newtheorem{theorem}{Theorem}[section]
\newtheorem{cor}[theorem]{Corollary}
\newtheorem{lem}[theorem]{Lemma}

\newtheorem{definition}[theorem]{Definition}

\newenvironment{rem}{\medskip\noindent{\it Remark:\/}}{\medskip}
\newenvironment{claim}{\medskip\noindent{\it Claim:\/}}{\medskip}

\title[Asymptotic expansion of the off-diagonal Bergman kernel]
{Asymptotic expansion of the off-diagonal Bergman kernel on compact K\"ahler manifolds}

\author{Zhiqin Lu}
\address{Department of Mathematics, UC Irvine, Irvine, CA 92697, USA}
 \email{zlu@uci.edu}

\author{Bernard Shiffman}
\address{Department of Mathematics, Johns Hopkins University, Baltimore, MD 21218, USA}
 \email{shiffman@math.jhu.edu}

\thanks{Research of the
first  author is partially supported by NSF grant  DMS-1206748.  Research of the
second  author is partially supported by NSF grant  DMS-1201372}

\begin{document}

\begin{abstract}  We compute the first four coefficients 
of the asymptotic off-diagonal expansion in \cite{SZ2} of the  Bergman kernel for the $N$-th power of a positive line bundle on a compact \kahler manifold, and we show that the coefficient $b_1$ of the $N^{-1/2}$ term vanishes when we use a  K-frame. We also show that all the coefficients of the expansion are polynomials in the K-coordinates and the covariant derivatives of the curvature and are homogeneous with respect to the weight $w$ in \cite{lu-1}. \end{abstract}

\maketitle


\section*{Introduction}

 The subject of this paper is the asymptotic expansion of the Bergman kernels for holomorphic sections of powers $L^{N}$ of a positive holomorphic line bundle $L$ on a compact \kahler manifold $M$.  The Bergman kernel $K_N(z,w)$ for $L^N\to M$ is the kernel 
 of the orthogonal projection from the space of $\lcal^2$ sections of
$L^{N}$ to the holomorphic sections.    In 1990, Tian \cite{Ti} gave the leading asymptotics of the diagonal kernel $K_N(z,z)$, and a complete  asymptotic expansion   was  given by Zelditch \cite{Z} and independently by Catlin \cite{Cat} using the Boutet de Monvel-Sj\"ostrand 
parametrix \cite{BouSj}.  A purely complex-geometric proof of the Catlin-Zelditch expansion is given in  \cite{zl}; see also \cites{BBS,DLM, MM} for other approaches. The first three terms of the diagonal expansion were computed in \cite{lu-1}; alternative approaches to the computation of the terms of the expansion are given in \cites{DK-1, haoxu}.   

Zelditch \cite{Z} obtained the asymptotic expansion of $K_N(z,z)$ by viewing it as a  \szego kernel as follows:  the holomorphic sections of $L^{ N}$ can be regarded as CR-holomorphic functions on a circle bundle  $X\to M$ (as described below), and the kernels $K_N(z,w)$ lift to form the reproducing kernels $\Pi_N(x,y)$ for the Fourier components $\hcal^2_N(X)$ of the Hardy space $\hcal^2(X)$  of CR-holomorphic functions on $X$.  Thus the  $\Pi_N(x,y)$ are the Fourier components of the \szego kernel for $X$.   

The Bergman projection kernel  $\Pi_N(z,w)$  can also be viewed as the covariance function (or two-point function) of a Gaussian random field  on $X$. It was shown in \cite{SZ3} that the zero correlation current (which in one complex dimension gives the pair correlation for the point process of zeros of Gaussian holomorphic functions or sections) can be given by a universal formula depending only on the covariance $\Pi_N(z,w)$.  This formula for the zero correlation current together with the scaled off-diagonal asymptotics of \cite{SZ2}  was applied   to the distribution of  zeros,  critical points and excursion sets  of  random holomorphic sections in \cites{Baber, Ha, SZ3,SZ4,SZZ, Sun}. For example,    \cites{SZ3,SZ4} generalize results of Sodin-Tsirelson \cite{ST} on the variance and asymptotic normality of the number of zeros of a random holomorphic function on a domain in $\C$; \cite{SZZ} generalizes \cite{ST3} on the ``hole probablity" of finding no zeros in a domain. 

 The off-diagonal  asymptotics  of the Bergman kernel have also found applications, for example, to  random Bergman metrics \cite{FKZ},  the convergence of geodesics in the space of Bergman metrics on toric varieties \cite{SoZe}, and   the asymptotics of spectral projectors associated to Toeplitz operators  \cite{Pa}. 
 
 A complete asymptotic  expansion of the scaled off-diagonal kernel on a compact, almost complex symplectic manifold $M$ was given in  \cite{SZ2}. In terms of normal coordinates in a point $z_0\in M$, this expansion takes the form  
\begin{equation}\label{near-intro}N^{-m}\,\Pi_N(\uN,\, \vN) \sim \frac {1}{\pi^m}\,e^{u\cdot v-\half(|u|^2+|v|^2)}(1+N^{-1/2}b_1+N\inv b_2 +\cdots N^{-r/2}b_r+\cdots)\,,\end{equation}  for $|u|+|v|<a$. (A precise statement  is given in Theorem \ref{near} below.) 

In this paper, we  consider  the integrable case  where  $M$ is a compact \kahler manifold.
We   
show that the coefficients $b_j$ of \eqref{near-intro}  are polynomials in  the ``K-coordinates"  and the covariant derivatives of the curvature\footnote{The curvature tensor itself is regarded as the $0$-th covariant derivative. We use this convention throughout the paper.}  (Theorem~\ref{homogeneity}), and we compute the first few terms  (Theorems~\ref{main2} and~\ref{mainresult}).   Our formula for computing the terms of the expansion is given in  Lemma~\ref{algorithm}.  We also show  (Lemma~2.5) that the first coefficient $b_1$ of  the expansion \eqref{near-intro} vanishes when we use a ``K-frame" for $L$ and normal coordinates  at $z_0$, and thus $N^{-m}\,\Pi_N(\uN,\, \vN)$ converges at the rate $1/N$ (instead of $1/\sqrtn$) as $N\to\infty$.  The vanishing of $b_1$ (using  different non-holomorphic coordinates and frame) was also given by Ma and Marinescu \cite[(4.1.26)]{MM}, \cite[(2.19)]{MM2}.     

The expansion \eqref{near-intro} immediately gives the $\ccal^k$-bounds
\begin{equation}\label{Ck} D^k\log \Pi_N(u,v)=O(N^{\frac{k-1}2})\,,\quad \mbox{for }\ |u|+|v|<\frac a\sqrtn\,,\ k\ge 3\,.\end{equation} 
(For $k=1,2$, the bound is $O(N^\frac k2)$.\,)
 As a consequence of the proof of our formulas for the coefficients of \eqref{near-intro},  we prove that one can replace the bounds in \eqref{Ck} with  the sharp bounds $O(\sqrtn)$ for $k=3$ and $O(N)$ for all $k\ge 4$ (Theorem~\ref{DKlog}).

\section {Statement of results}
 Let $M$ be a compact complex manifold polarized by an ample line bundle $L$.  To describe the scaled  off-diagonal asymptotic expansion, we
give $L$ a hermitian metric $h$ with positive curvature $\Theta_h$, and we give $M$ the \kahler
form $\om= \frac i2\Theta_h$.
We recall that the Hermitian metric $h$ on $L$ induces
Hermitian metrics
$h^N$ on
$L^N$, and we have $$c_1(L^N,h^N) = N\,c_1(L,h)= \frac N\pi \,\om\,.$$
The  metrics
$h,\om$ induce  Hermitian inner products on the (finite-dimensional) spaces 
 $H^0(M,L^N)$ of holomorphic sections of $L^N$.

 Let $X$ be the unit circle bundle of  the dual line bundle $L^*$. 
As in \cites{BSZ2, SZ2,Z}, we lift holomorphic sections $s\in H^0(M,L^N)$ to $CR$  functions $\hat s$
on $X$ satisfying  $\hat s(e^{i\theta}x)= e^{iN\theta}\hat s(x)$.
We denote the space of such functions by  $ \hcal^2_N(X)$. The
 Bergman   kernel for $H^0(M,L^N)$ lifts to the orthogonal projector
$\Pi_N:\lcal^2(X)\to\hcal^2_N(X)$, which is given by the {\it
\szego kernel}
$$\Pi_N(x,y)=\sum_{j=1}^{d_N} \wh S^N_j(x)\overline{\wh S^N_j(y)}
\qquad (x,y\in X)\;,$$ where the  $\wh S^N_j$ form an orthonormal basis for $\hcal^2_N(X)$,  and 
$d_N=\dim H^0(M,L^N)$. 

 We are interested in the scaled off-diagonal asymptotics of the \szego kernel in a neighborhood of a point $z_0\in M$.
We let  $z= (z_1,\dots,z_m)$ denote  local complex coordinates in a neighborhood $U$ of  $z_0$, and we let $|z|=\sqrt{|z_1|^2+\cdots+|z_m|^2}$.  
Throughout this paper, we shall  assume that  the coordinates of $z_0$ are $(0,\dots,0)\in\mathbb C^m$.
To properly describe the scaling asymptotics for the \szego kernel at a point  $z_0
\in M$,   we need to use suitable coordinates in the circle bundle $X$. To do this, we first recall the following definition from \cite{SZ2}:

\begin{definition}\label{preferred} A preferred local frame at a point $z_0\in M$ is a $\ccal^\infty$ section $\e_L$  of $L$  over a neighborhood $U$ of $z_0$ 
such that $\dbar \e_L$ vanishes to first order at $z_0$ and 
\begin{equation} \phi(z):=-\log \|\e_L(z)\|^2_h =  \sum g_{j\bar k}(0)z_j\bar z_k +O(|z|^3)\,.\label{defphi}\end{equation}
\end{definition}

 Suppose that $(z_1,\dots,z_m)$ are  complex  coordinates on a neighborhood $U$ of $z_0$ and $\e_L$ is a $\mathcal C^\infty$  local frame over $U$. We then give points $e^{i\theta-\phi(z)/2}\,\e_L^*(z)$ of the circle bundle $X$ (over $U$) the coordinates $(z_1,\dots,z_m,\theta)\in \C^m\times (\R/2\pi\Z)$, and we write
\begin{equation}\label{coords}\Pi_N(z,\theta_1; w,\theta_2)=\Pi_N^{z_0}(z,\theta_1; w,\theta_2)
=\Pi_N\left(e^{i\theta_1-\phi(z)/2}\,\e_L^*(z)\;,\;e^{i\theta_2-\phi(w)/2}\,\e_L^*(w)\right)\;.\end{equation}  

We also let  \begin{equation}\label{BF}\Pi_{BF}\left(z,\theta_1;w,\theta_2\right) =
\frac{1}{\pi^m} e^{i  (\theta_1 - \theta_2)+  z \cdot \bar{w} - \frac
12 (|z|^2+ |w|^2)}\end{equation} denote the \szego kernel for the Bargmann-Fock space
 of functions on $\C^m$  (see \cite{BSZ2}).

We recall the  off-diagonal
asymptotics from \cite {SZ2} (see also \cite{SZ3}*{\S 5}):

\begin{theorem}\label{near}{\rm (cf.\ \cite{SZ2}*{Theorem~3.1})} Let $(L,h)\to (M,\om)$ be a positive line bundle over an $m$-dimensional compact \kahler manifold with  \kahler
form $\om= \frac i2\Theta_h$. Let $\e_L$ be a   preferred  local frame for $L$ and let $z_1,\dots,z_m$ be   complex coordinates  about a point $z_0\in M$ such that the \kahler potential $\phi  :=-\log \|\e_L(z)\|^2_h = |z|^2 +O(|z|^3)$.  Then  for all $k\in\Z^+$, 
\begin{align} 
\begin{split}\label{b_r} & N^{-m}\,\Pi_N^{z_0}\left(\frac{u}{\sqrtn},\frac{\theta_1}{N};
\frac{v}{\sqrtn},\frac {\theta_2} N\right)\\
&= \Pi_{BF}(u,\theta_1;v,\theta_2)\left[1+ \sum_{r = 1}^{k} N^{-r/2}
b_{r}(u,v) + N^{-(k +1)/2}  E_{Nk}(u,v)\right]\;,\end{split}\end{align}
where  \begin{itemize}
\item each
$b_r(u,v)$ is a polynomial  (in $u,v,\bar u,\bar v$)  of degree at most $5r$,    and is an even polynomial if $r$ is even, and odd if $r$ is odd; 
\item for all $a\in\R^+$ and  $j\ge 0$, there exists a positive constant $C_{jka}$ such that 
\begin{equation}\label{ENk}|{  D^j} E_{Nk}(u,v)|\le C_{jka}\quad \mbox{for }\
|u|+|v|< a\,.\end{equation}\end{itemize}\end{theorem}
Here   $|D^jF(u,v)|$ denotes the sum of the norms of  the partial derivatives of $F$ of order $j$ at  $(u,v)$.

If one chooses a smoothly varying family of normal coordinates and preferred frames about points $z$ in a neighborhood of $z_0$, then the polynomials $b_k$ and remainders $E_{Nk}$ in~\eqref{b_r} vary smoothly in $z$, and the $(z,u,v)$-derivatives are bounded independently of $N$; setting $u=v=0$ in \eqref{b_r}, one obtains the Catlin-Zelditch asymptotic expansion \cites{Cat,Z},
\begin{equation}\label{CZ}  N^{-m}\,\Pi_N(z,z) \sim \frac {1}{\pi^m}\left[ 1 + a_1(z)N\inv +a_2(z)N^{-2} +\cdots +a_r(z)N^{-r}+\cdots\right]\,.\end{equation} Precisely, for all $j,k\ge 0$, there exist constants $C_{jk}$ such that $$\left\| N^{-m}\,\Pi_N(z,z)- \pi^{-m}[ 1 + a_1(z)N\inv  +\cdots +a_k(z)N^{-k}]\right\|_{\ccal^j(M)} \le N^{-k-1}C_{jk}.$$

 Note  that we do not assume in Theorem \ref{near} that the coordinates $z_j$ are holomorphic.  (Indeed,  Theorem~\ref{near}  is stated in \cite{SZ2}    for almost complex, symplectic manifolds.) Note also  that the point in $M$ with coordinates $(u_1/\sqrtn,\dots,u_m/\sqrtn)$ depends  on the $\ccal^\infty$ germ at $z_0$ of the coordinate map,  and thus the polynomials $b_r(u,v)$ are not well-defined functions on $T_{z_0}(M)$, but depend on the higher jets of the coordinate map  and the local frame. However, in our main results (Theorems~\ref{homogeneity} and \ref{mainresult}),  we use K-coordinates and a K-frame, which we define below. By the uniqueness properties of these coordinates (see the discussion  following Definition~\ref{Bochnerinfty}),   our coefficients are intrinsic.

Let $z_1,\cdots,z_m$ be holomorphic local coordinates about  $z_0$ and let $\e_L$ be a holomorphic frame about  $z_0$. Then  the \kahler form is given by  \begin{equation}\label{kahler} \om = \frac i2 \ddbar \phi = \frac i2 g_{j\bar k}\,d z_j\wedge d\bar z_k\,.\end{equation}  
The components of the curvature tensor
are given by 
\begin{equation}\label{Rijkl}
R_{i\bar jk\bar l}=-\frac{\pa^2 g_{i\bar j}}{\pa z_k\pa \bar z_{l}}+g^{p\bar q}\,\frac{\pa g_{i\bar q}}{\pa z_k}\,\frac{\pa g_{p\bar j}}{\pa\bar  z_{l}}\,.
\end{equation}   (Throughout this paper,   formulas are summed over repeated indices.)
 By our convention, the scalar curvature is  $\rho=R_{i\ j}^{\ i\ j}
=g^{i\bar j}g^{k\bar l}R_{i\bar jk\bar l}$\,. 

The above formulas hold at $z_0$ under the following weaker conditions. 
 If $\e_L$ is a  preferred frame at $z_0$, then \eqref{kahler} holds  at $z_0$; if $\e_L$ is holomorphic to infinite order at $z_0$,  i.e.\ $\bar\pa\e_L$ vanishes to infinite order at $z_0$,   then \eqref{kahler} holds  to infinite order at $z_0$. 

If the complex coordinates $z_1,\dots,z_m$ are {\it holomorphic to infinite order at\/} $z_0$, i.e.\ the 1-forms $\dbar z_j$ vanish to infinite order at $z_0$, then  \eqref{kahler} and \eqref{Rijkl}   hold to infinite order at $z_0$.

A consequence of our  formula for the coefficients of the expansion (Theorem~\ref{mainresult}) is that $b_1$ vanishes and $b_2$ is an expression in the curvature tensor when we use  appropriate  coordinates and  frame: 
 
 \begin{theorem}\label{main2} Let $(L,h)\to (M,\om)$  with $\om = \frac i2\ddbar\phi$  be as in Theorem \ref{near}.  Let $\e_L$ be a holomorphic frame and  let $z_1,\dots,z_m$ be holomorphic coordinates such that the potential $\phi$ is of the form \begin{equation} \label{Korder4}\phi(z) = |z|^2 + \left.\frac 14 \,\frac{\d^4\phi}{\d z_j\d z_p\d\bar z_k\d\bar z_l}\right|_{z=0} z_j\bar z_k z_p\bar z_q + O(|z|^5)\,.\end{equation}  Then
\begin{multline*}   N^{-m}\,\Pi_N^{z_0}\left(\frac{u}{\sqrtn},\frac{\theta_1}{N};
\frac{v}{\sqrtn},\frac {\theta_2} N\right)\\
= \Pi_{BF}(u,\theta_1;v,\theta_2) \Big[1+ N\inv\Big(\textstyle\half\rho+\frac 18 R(u,\bar u,u,\bar u) + \frac 18 R(v,\bar v,v,\bar v) -\frac 14 R(u,\bar v,u,\bar v)\Big)\\ +  N^{-3/2}  E_{N3}(u,v)\Big]\:,\end{multline*}
where $E_{N3}$ satisfies the estimate \eqref{ENk}  for $k=3$.
\end{theorem}

Here, $R$ denotes the curvature tensor $R(s,\bar t, u,\bar v)=R_{i\bar j k\bar l}(z_0) s_i \bar t_j u_k \bar v_l$ and $\rho =\rho(z_0)$ is the scalar curvature at $z_0$.  Under the assumption of the theorem, the coordinates are normal coordinates at $z_0$, and therefore $$R_{j\bar kp\bar q}(0) = -\frac{\pa^2 g_{j\bar k}}{\pa z_p\pa \bar z_{q}}(0)=-\frac{\d^4\phi}{\d z_j\d z_p\d\bar z_k\d\bar z_l}(0)\,.$$  Condition \eqref{Korder4} means that  $z_1,\dots,z_m$ and $\e_L$ are ``K-coordinates with a K-frame of order 4" as in Definition \ref{bochnerdef}. (See the discussion following Definition \ref{Bochnerinfty} for the existence of holomorphic K-coordinates and K-frames.)   The terms of the off-diagonal expansion  of the \szego kernel  depend on the choice of coordinates and are in general not tensors. After we  posted the first version of this paper, Xiaonan Ma and George Marinescu applied their prior work \cites{MM,MM2} to compute the $N^{-1}$ term of an off-diagonal expansion using  different (non-holomorphic) coordinates and frame (see \cite{MM3}). 

Setting $v=0,\ \theta_2=0$ in Theorem \ref{main2}, we have:

\begin{cor}\label{main2cor} Under the hypotheses of Theorem \ref{main2}, the off-diagonal \szego kernel at the pair of points $z_0+u/\sqrtn$ and $z_0$ is given by
\begin{equation*}   N^{-m}\,\Pi_N^{z_0}\left(\frac{u}{\sqrtn},\frac{\theta}{N};
0,0\right)=\frac{1}{\pi^m} \, e^{-\frac12|u|^2+i\theta} \left[1+ N\inv\Big(\textstyle\half\rho+\frac 18 R(u,\bar u,u,\bar u)\Big) +  O(N^{-3/2})\right]\end{equation*} for $|u|<a$,  where $a\in\R^+$. \end{cor}

Recall that  $R(u,\bar u,u,\bar u)$ is the holomorphic sectional curvature in the $u$-direction multiplied by $|u|^4$.

In order for the higher $b_j$ to be well-defined, we need to strengthen the hypotheses on the coordinates and local frame.  To do this we first summarize \cite{zl}*{Definition A.1, A.2} as follows:

\begin{definition} \label{bochnerdef} Let $\e_L$ be a holomorphic frame of $L$ about  $z_0$ and $z_1,\cdots, z_m$ be a holomorphic coordinate system about  $z_0$. Let $\phi(z)$ be the \kahler potential function  given by \eqref{defphi}. 
Let $$\phi(z) \sim |z|^2 + \sum a_{JK} z^J \bar z^K\,,\quad |J|+|K| \ge 3$$ be the   Taylor expansion of the  \kahler potential, where we use the notation $J=(j_1,\dots, j_m)$, $|J|= j_1+\cdots+j_m$, $z^J=z_1^{j_1}\cdots z_m^{j_m}$. 
Let $3\le n\le \infty$.  The  frame 
$\e_L$ is called a 
K-frame of order $n$ if  $a_{JK}=0$  whenever $|J|+|K|\le n$ and either $|J|=0$ or $|K|=0$.   The coordinates $z_1,\cdots, z_m$ are called the 
K-coordinates of order $n$, if $a_{JK}=0$  whenever $|J|+|K|\le n$ and either $|J|=1$ or $|K|=1$.
\end{definition}
 
 We recall that the notation $f(z) \sim  \sum c_{JK} z^J \bar z^K$ for an asymptotic series means that $f(z) =  \sum _{|J|+|K| \le n}c_{JK} z^J \bar z^K +O(|z|^{n+1})$, for all $n\in \Z^+$.  

Combining the two parts of Definition \ref{bochnerdef}, we have K-coordinates with a K-frame of order $n$ precisely when 
\begin{equation} \label{bochnerk}  \phi(z) = |z|^2 +\!\!\! \sum_{\begin{smallmatrix}|J|\ge 2, \; |K| \ge 2\\|J|+|K|\le n\end{smallmatrix}}\!\!\! a_{JK} z^J \bar z^K \;+\;O(|z|^{n+1})\,.\end{equation}   
In particular, normal coordinates are K-coordinates of order 3;  a preferred frame  is a  K-frame of order 2.

 If $\omega$ is only smooth, one cannot always find K-coordinates or  a K-frame   of  order $n=\infty$ in the above definition, so we extend our definition of  K-coordinates and a K-frame for this case: 

\begin{definition} \label{Bochnerinfty} Let $\e_L$ be a smooth frame about  $z_0$ such that $\dbar \e_L$ vanishes to infinite order at a point $z_0$, and let $z_1,\dots,z_m$ be complex coordinates about $z_0$  such that $\dbar z_j$ vanishes to infinite order at  $z_0=0$.  We say that these are  K-coordinates with a K-frame at $z_0$ if the Taylor expansion of the \kahler potential has the form
 \begin{equation}\label{bochner}\phi(z) \sim |z|^2 + \sum a_{JK} z^J \bar z^K\,,\quad |J|\ge 2\,, \ |K| \ge 2\,,\end{equation}   \end{definition}

Bochner \cite{Bo} showed that if $\phi$ is real-analytic, then (after adding a pluriharmonic  function to $\phi$, i.e.\  by making a holomorphic change of the local frame $\e_L$)  there exist unique holomorphic  K-coordinates of infinite order  (up to  a unitary  transformation) and a  unique holomorphic K-frame of infinite order (up to a unit complex number),  and the   expansion \eqref{bochner} is in fact a convergent power series equal to $\phi(z)$. When the \kahler form $\omega$ is $\ccal^\infty$,  one can find holomorphic K-coordinates  and a holomorphic K-frame of any finite order $n$. Furthermore, in this case,  one  can use  Bochner's  method to find formal power series solutions to \eqref{bochner}. Using the fact that any formal power series is the asymptotic Taylor series of a $\ccal^\infty$ function, one obtains the  existence of smooth K-coordinates with a K-frame (of infinite order) as in Definition \ref{Bochnerinfty}.

 The second main result of this paper, Theorem \ref{homogeneity} below, is that the coefficients $b_j$ of the asymptotic expansion \eqref{b_r} are ``homogeneous" expressions in the curvature tensor. To describe what this means, we let $\mathcal F$ 
be the ring of polynomials of the covariant derivatives of the curvature  with coefficients in $\Q$.   
A monomial in $\mathcal F$ is defined to be  a product  of terms like $\rho_{,I\bar J}$, ${\rm Ric}_{i\bar j,I\bar J}$, or $R_{i\bar jk\bar l,I\bar J}$, where $I=(i_1\cdots i_p)$ and $J=(j_1\cdots j_q)$  are multiple indices.
In~\cite{lu-1},  the weight  $w$ is defined by:
\begin{equation}\label{weight}
w(\rho_{,I\bar J})=w({\rm Ric}_{i\bar j,I\bar J})=w(R_{i\bar jk\bar l,I\bar J})=1+\frac{p+q}{2}.
\end{equation}

We let $\mathcal P$ be the space of polynomials of $u,v,\bar u,\bar v$ with coefficients in $\mathcal F$, and we extend the definition of $w$  to   monomials in $\pcal$  by requiring that $w(AB)=w(A)+w(B)$ and $w(u)=w(v)= w(\bar u)=w(\bar v)=0$.
We call an element in $\mathcal P$  $w$-homogeneous  if all of its monomials have the same weight. If an element is $w$-homogeneous, we define its degree  to be the weight of its monomials.

\begin{theorem} \label{homogeneity} Let $2\le r\le n-2$.  Then under the hypotheses of Theorem \ref{near} with  K-coordinates $(z_1,\dots,z_m)$ and K-frame  $\e_L$ of order $n$, the
 coefficient $b_r$ is a   polynomial in $u,v,\bar u,\bar v$ and the covariant derivatives of the curvature, and  is $w$-homogeneous of weight  $r/2$.  Moreover, as a polynomial  in   $u,v,\bar u,\bar v$ (with coefficients in the ring $\fcal$ defined above):  
\begin{itemize} \item if $r$ is even, $b_r$ is  an even polynomial of degree $2r$  in   $u,v,\bar u,\bar v$; 
\item  if $r$ is odd, $b_r$ is  an odd polynomial   of degree  $ 2r-1$  in   $u,v,\bar u,\bar v$.
\end{itemize}
 \end{theorem}

\begin{rem}   Theorem 3.1 in  \cite{SZ2} says that $b_r$ is a polynomial of the same parity as $r$, but  gives the  bound $\deg b_r\le 5r$ for the more general case of almost complex symplectic manifolds  (see Theorem~\ref{near} above). 
 We note that  Theorem~\ref{homogeneity} states that  $b_r$ has exactly the degree $2r$ ($2r-1$, respectively) for $r$ even (odd, respectively), when the terms $\rho_{,I\bar J}$, ${\rm Ric}_{i\bar j,I\bar J}$,  $R_{i\bar jk\bar l,I\bar J}$, etc., are regarded as abstract coefficients. When  they are evaluated at a point $z_0\in M$, then $2r, 2r-1$ are only upper bounds for the degree of $b_r$.  For example, if $M$ is a polarized abelian variety with the flat \kahler metric, then all the $b_r$ vanish. \end{rem}

\section {Computation of the coefficients}\label{theproofs}

In this section, we prove Theorem \ref{main2} and also give formulas for the coefficients $b_3$ and $b_4$. 
  We begin by describing notation that we use for these formulas.  Let $z_1,\cdots, z_m$ be K-coordinates  at $z_0\in M$. Writing 
$$\begin{array}{l} \bu=\sum u_j\, \d/\d z_j|_{z_0} \in T^{1,0}_{z_0},\quad \overline\bu=\sum \bar u_j\, \d/\d\bar z_j|_{z_0} \in T^{0,1}_{z_0}\,,\\[8pt] 
\bv=\sum v_j\, \d/\d z_j|_{z_0} \in T^{1,0}_{z_0},\quad \overline\bv=\sum \bar v_j\, \d/\d\bar z_j|_{z_0} \in T^{0,1}_{z_0}\,, \end{array}$$
we let 
\begin{align}\label{P} 
\begin{split}
S(\bu,\bar \bv ) & = -R(\bu,\bar\bv,\bu,\bar\bv)=-R_{i\bar jk\bar l}u_iu_k\bar v_j\bar v_l\,,
\\
L(\bu,\bar\bv)&=-R_{i\bar jk\bar l,s} u_iu_ku_s\bar v_j\bar v_l - R_{i\bar jk\bar l,\bar s} u_iu_k\bar v_s\bar v_j\bar v_l,
\\
K_1(\bu,\bar\bv)&=(-R_{{i\bar jk\bar l},s\bar t} + R_{p\bar jk\bar l} R_{i\bar p s\bar t}+ R_{i\bar jp\bar l} R_{k\bar ps\bar t}+ R_{i\bar pk\bar t} R_{p\bar js\bar l})u_iu_ku_s \bar v_j\bar  v_l\bar v_t\,,
\\
K_2(\bu,\bar\bv)&= -R_{i\bar jk\bar l,st}u_iu_ku_su_t\bar v_j\bar v_l
-R_{i\bar jk\bar l,\bar s\bar t}u_iu_k\bar v_j\bar v_l\bar v_s\bar v_t.
\end{split}
\end{align}
Recall that the covariant derivatives of tensors,  such as those appearing in  \eqref{P},  are defined as
\begin{align*}
&T_{i_1\cdots i_p\bar j_1\cdots \bar j_q,s}=\frac{\pa T_{i_1\cdots i_p\bar j_1,\cdots \bar j_q} }{\pa z_s}-\sum_{t=1}^p\Gamma_{i_t s}^r
T_{i_1\cdots r\cdots i_p\bar j_1\cdots \bar j_q},\\
&T_{i_1\cdots i_p\bar j_1\cdots \bar j_q,\bar s}=\frac{\pa T_{i_1\cdots i_p\bar j_1\cdots \bar j_q} }{\pa\bar  z_s}-\sum_{t=1}^q\overline{\Gamma_{j_t s}^r}
T_{i_1\cdots i_p\bar j_1\cdots  \bar r\cdots  \bar j_q}.
\end{align*}

For any function  $f:T^{1,0}_{z_0}\times T^{0,1}_{z_0} \to \C$, we let $f^\sharp:T^{1,0}_{z_0}\times T^{1,0}_{z_0} \to \C$ be the function given by
\begin{equation}\label{fsharp} 
f^{\sharp}(\bu,\bv)=f(\bu,\overline\bv)- \half f(\bu,\overline \bu)-\half f(\bv,\overline\bv)\,,\qquad \bu,\bv\in T^{1,0}_{z_0}\,.
\end{equation}

 We now can state our formulas for the first 4 coefficients of the asymptotic expansion of Theorem \ref{near}:
 
\begin{theorem}\label{mainresult} Let $(L,h)\to (M,\om)$ be a positive line bundle over a compact \kahler manifold with  \kahler
form $\om= \frac i2\Theta_h$. Let   $z_1,\dots,z_m$ be K-coordinates with a  K-frame $\e_L$ at  a point $z_0\in M$ (as in Definition \ref{Bochnerinfty}). 

Then the first four coefficients in the off-diagonal asymptotic expansion \eqref{b_r} are given by
\begin{align*}
 b_1(u,v)&=0,\\  
 b_2(u,v)&=\frac 12\rho+\frac 14 S^\sharp(\bu,\bv),\\
  b_3(u,v)&=\frac 12\nabla\rho(\bu+\overline\bv)+\frac 1{12}L^\sharp(\bu,\bv),\\
  b_4(u,v)&=\frac 13\Delta\rho+\frac{1}{24}(|R|^2-4|{\rm Ric}|^2)+\frac{1}{4}\nabla^2\rho(\bu+\overline\bv,\bu+\overline\bv),\\
&\qquad\quad\qquad +\frac {1}{36} K_1^\sharp(\bu,\bv)+\frac {1}{48} K_2^\sharp(\bu,\bv)
+\frac 18\left(\rho+\frac 12S^\sharp (\bu,\bv)\right)^2,\end{align*}where $\rho,\ {\rm Ric},\ R$ are the scalar curvature, Ricci curvature tensor, and the curvature tensor, respectively, at $z_0$, and
 $S^\sharp, L^\sharp, K_1^\sharp, K_2^\sharp$ are given by \eqref{P}--\eqref{fsharp}.
Furthermore, the formula for $b_j$ holds when $z_1,\dots,z_m$ are K-coordinates with a K-frame $\e_L$ of order $j+2$ (as in Definition \ref{bochnerdef}), for $1\le j\le 4$.
\end{theorem}
 
 The values of $b_1,b_2$ in Theorem \ref{mainresult} yield the formula of Theorem \ref{main2}. Theorem \ref{mainresult} also provides scaling asymptotics of the {\it normalized \szego kernel} in \cite{SZ3}:
\begin{equation}\label{PN} P_N(z,w):=
\frac{|\Pi_N(z,0;w,0)|}{\Pi_N(z,0;z,0)^\frac 12\, \Pi_N(w,0;w,0)^\frac
12}\;.\end{equation}

\begin{cor} \label{PNasymp} Under the hypotheses and notation of Theorem \ref{mainresult}, 
$$P_N\left(\frac{u}{\sqrtn},
\frac{v}{\sqrtn}\right)= e^{-\half|u-v|^2}\left[1+\frac 14 \Re S^\sharp(\bu,\bv)\,N\inv +\frac 1{12}\Re L^\sharp(\bu,\bv)\,N^{-3/2} +O(N^{-2})\right].$$ 
\end{cor}

 We begin the proof of Theorem \ref{mainresult} by introducing some more notation: 

\begin{definition} \label{BN}For a local  smooth frame $\e_L$ of $L\to M$ over a trivializing neighborhood $U$,   we define  the kernel $B_N(z,w)$ on $U\times U$ by
\[
\Pi_N(z,\theta_1;w,\theta_2) = e^{iN(\theta_1-\theta_2)} e^{-\frac N2\phi(z)-\frac N2\phi(w)} B_N(z, w)\,.
\]
\end{definition}
Note that if we write $S^N_j= f_j \e_L^N$, where  $\{S^N_j\}$ is an orthonormal basis for $H^0(M,L^N)$, then
\begin{equation}\label{BN-2}
B_N(z,w)= \sum f^N_j(z) \overline{f^N_j(w)}\;. 
\end{equation}
In particular, $B_N(z,w)$ is independent to $\theta_1,\theta_2$.

 If $\e_L$ is holomorphic, then the $f_j^N$ are holomorphic and thus  $B_N(z,w)$ is holomorphic in $z$ and anti-holomorphic in $w$. More generally, 
 if $\dbar \e_L$ vanishes to infinite order at $z_0$, then $B_N(z,w)$ is holomorphic in $z$ and anti-holomorphic in $w$ to infinite order at $(z_0,z_0)$.

 To simplify notation, we now write \begin{eqnarray*} &&\Pi_N(z,w)=\Pi_N^{z_0}(z,0;w,0)\,,\quad i.e. \ \ \Pi_N^{z_0}(z,\theta_1;w,\theta_2)=
e^{iN(\theta_1-\theta_2)}\,\Pi_N(z,w)\,,\\ &&
\phi_{P\overline Q}(z) = \frac{\pa^{|P|+|Q|}\phi(z)}{\pa z^P\pa\bar z^Q}\,.\end{eqnarray*}
For a multi-index $P=(p_1,\cdots, p_m)$, we define $P!=p_1!\cdots p_m!$. 

\begin{lem}\label{taylor}  Suppose that $(z_1,\dots,z_m)$ are normal coordinates  holomorphic to infinite order  at  $z_0$, and $\e_L$ is a preferred frame  holomorphic to infinite order  at   $z_0$.  
 
Then  for all $k\in\Z^+$ and $ a\in\R^+$, there exists a positive constant $C_{ka}$ such that 
\begin{align}
\log\Pi_N(\uN,\vN) = & \sum_{ |P|+|Q|\le k-1} \frac{1}{P!Q!}N^{-\frac{|P|+|Q|}{2}}\left.\frac{\pa^{|P|+|Q|}\log\Pi_N(z, z)}{\pa z^P\pa\bar z^Q}\right|_{z=0} u^P\bar v^Q\notag\\ &\   
+ \sum_{|P|+|Q|\le  k+1}\frac{1}{P!Q!}N^{-\frac{|P|+|Q|}{2}+1}\phi_{P\overline Q}(0)\left[u^P\bar v^Q-\frac12 u^P\bar u^Q-\frac12 v^P\bar v^Q\right]\notag\\&\ 
+\ N^{-\frac{k}2}F_{Nk}(u,v) \,,\label{logPiN0}
\end{align}
 where $| F_{Nk}(u,v)|\le C_{ka}$  for $|u|+|v|<a$, for all $N>0$.
\end{lem}

\begin{proof} By Definition \ref{BN}  and the simplified notation after the definition, 
\begin{equation}\label{PiB}
\log\Pi_N(\uN,\vN)= \log B_N(\uN,\vN)\,{-\frac N2\phi(\uN)-\frac N2\phi(\vN)}.
\end{equation} 
By~\eqref{BN-2}, $B_N(u,v)$ is holomorphic  in  $u$ and anti-holomorphic in $v$  to infinite order at  $(0,0)$, and therefore  by  Taylor's formula, 
\begin{multline}\label{logBN}
\left|\log B_N(\uN,\vN)-\sum_{|P|+|Q|\leq 5k-1}\frac{1}{P!Q!}N^{-\frac{|P|+|Q|}{2}}\left.\frac{\pa^{|P|+|Q|}\log B_N(z, z)}{\pa z^P\pa\bar z^Q}\right|_{z=0} u^P\bar v^Q\right|\\
\leq C_{mk}\,\left(\frac {a}\sqrtn\right)^{ 5k}\sup_{|u|+|v|<\frac {a}\sqrtn}|D^{ 5k}\log B_N(u,v)|.
\end{multline} 

Next, we estimate $D^k \log B_N(u,v)$:  by Theorem \ref{near},
\begin{equation}\label{logPi}
\log\Pi_N(\uN,\vN) = u\cdot \bar v - \half |u|^2-\half |v|^2 +\sum_{j=1}^{k-1} N^{-j/2} \beta_j  (u,v) +N^{-k/2}{\wt  F}_{Nk}(u,v)\,,\end{equation} where $\beta_1,\dots,\beta_{k-1}$ are polynomials in $(u,\bar u,v,\bar v)$, and $|D^j {\wt F}_{Nk}(u,v) |<\wt C_{jka}$ for $|u|+|v|<a$.  

\begin{claim} \  $\deg \beta_j\le 5j$. \end{claim} 

To verify the claim, we recall from Theorem \ref{near} that $\deg b_r\le 5r$.  Assigning the weights $\nu(N^{-1/2}) = -5$, $\nu(u)=\nu(\bar u)= \nu(v)=\nu(\bar v)=1$, and extending the definition of $\nu$ to monomials in $N^{-1/2},u,\bar u,v,\bar v$ by requiring that $\nu(AB)=\nu(A)+\nu(B)$, we see that all the monomials in \eqref{b_r} have $\nu$-weights $\le 0$. Since the terms $N^{-j/2}\beta_j$ 
in \eqref{logPi} are polynomials in  the $N^{-r/2}b_r$,  the monomials in $N^{-j/2}\beta_j$ also have $\nu$-weights $\le 0$, which verifes the claim.

Since $u\cdot \bar v - \half |u|^2-\half |v|^2 +\sum_{j=1}^{k-1} N^{-j/2} \beta_j  (u,v)$ is a polynomial of degree $\le 5k-5$, we change variables from $(\uN,\vN)$ to $(u,v)$ in \eqref{logPi} to conclude that 
\begin{equation}\label{D5k} |D^{5k}\log \Pi_N(u,v)|=N^{-k/2}|D^{5k} \wt F_{Nk}(\sqrtn\,u,\sqrtn\,v)| \le\wt C_{[5k]ka} N^{2k}\,,\quad \mbox{for }\ |u|+|v|<\frac a\sqrtn\,. \end{equation}

By \eqref{PiB}, $\log B_N(z,z)=\log\Pi_N(z,z)+N\phi(z)$. 
Therefore by \eqref{logBN}  and ~\eqref{D5k},
\begin{align}\label{logBN2}
\log\, & B_N(\uN,\vN)\notag\\&=\sum_{|P|+|Q|\leq 5k-1}\frac{1}{P!Q!}N^{-\frac{|P|+|Q|}{2}}\left.\frac{\pa^{|P|+|Q|}\log B_N(z, z)}{\pa z^P\pa\bar z^Q}\right|_{z=0} u^P\bar v^Q+ N^{-k/2}{ E^1_{Nka}}(u,v)\notag\\
&=\sum_{|P|+|Q|\leq 5k-1}\frac{1}{P!Q!}N^{-\frac{|P|+|Q|}{2}}\left[\left.\frac{\pa^{|P|+|Q|}\log \Pi_N(z, z)}{\pa z^P\pa\bar z^Q}\right|_{z=0} +N\phi_{P\overline Q}(0)\right] u^P\bar v^Q\notag\\  
&\qquad\qquad + N^{-k/2}{ E^1_{Nka}}(u,v)\,,\end{align}
where $|{ E^1_{Nka}}(u,v)|\le C_{ka}$  for $|u|+|v|<a$.  Furthermore, by \eqref{PiB} and Taylor's formula,
\begin{align}\label{logPiN} 
&\log\, \Pi_N(\uN,\vN)\notag= \log B_N(\uN,\vN) \\
&-\frac N2 \sum_{|P|+|Q|\le k+2}N^{-\frac{|P|+|Q|}{2}}\phi_{P\overline Q}(0)(u^P\bar  u^Q +v^P\bar v^Q)  +N^{-\frac{k+1}2}{ E^2_{Nka}}(u,v),\end{align}
where $|{E^2_{Nka}}(u,v)|\leq C_{ka}$ for $|u|+|v|<a$. 

Since the derivatives of $\log\Pi_N$ (of order $>0$)  converge uniformly to 0 on the diagonal (by the Catlin-Zelditch asymptotic expansion \eqref{CZ}), each of the terms $\frac{\pa^{|P|+|Q|}\log \Pi_N(z, z)}{\pa z^P\pa\bar z^Q}|_{z=0}$ is bounded independently of $N$. Hence, we can move the $O(N^{-k/2})$ terms in the summation in \eqref{logBN2} to the remainder.  Applying \eqref{logPiN}, we then obtain the formula of the lemma. 
\end{proof}

 We now give a quick proof that the first term $b_1$ of the expansion \eqref{b_r} vanishes.  (This fact also follows from  formula \eqref{beta} below  for the coefficients $\beta_1$ of \eqref{logPi}.) 
\begin{lem}\label{lem25}  If $(z_1,\dots,z_m)$ are normal coordinates about $z_0$ and $\e_L$ is a K-frame of order 3, then 
\[
b_1(u,v)=0.
\]
\end{lem}

\begin{proof}  Under the assumptions on the coordinates  and  frame,  we  have
\[
\phi(z,\bar w)=z\cdot\bar w+O(|z|+|w|)^4.
\]
Thus
 by Lemma \ref{taylor} with $k=2$,
\begin{eqnarray*}\log\Pi_N(\uN,\vN) &= &  \log \Pi_N(0,0)   + \left.\frac 1{\sqrtn}\left(u_j \frac \d{\d z_j} +\bar v_j \frac\d{\d\bar z_j}\right)\log \Pi_N(z,z)\right|_{z=0}\\ && +  u\cdot \bar v-\half |u|^2-\half |v|^2  +O(\frac 1N)\,.\end{eqnarray*}

  On the other hand, by \eqref{BF}--\eqref{b_r}
 $$\log\Pi_N(\uN,\vN) = m\log  \frac N\pi + u\cdot \bar v-\half |u|^2-\half |v|^2 +N^{-1/2}b_1(u,v)+O(\frac 1N)\,.$$

But by the Catlin-Zelditch asymptotics  \eqref{CZ}, $d\log\Pi_N(z,z)=O(\frac 1N)$, and therefore $b_1=0$. \end{proof}
 Taking the logarithm of \eqref{CZ}, we obtain a $\ccal^\infty$  asymptotic expansion of the form
\begin{equation}\label{alpha} \log\Pi_N(z,z) \sim m\log(N/\pi) +\al_1(z)N\inv +\al_2(z)N^{-2} +\cdots +\al_r(z)N^{-r}+\cdots\,.\end{equation}
Substituting this expansion in \eqref{logPiN0} and equating coefficients  with those in \eqref{logPi},  we obtain: 
 
\begin{lem}\label{algorithm} The coefficients $\beta_j$  of \eqref{logPi}  ($j\ge 1$)  are polynomials of degree at most $j+2$ given by
\begin{eqnarray} \beta_j(u,v)&=& \sum_{r=1}^{\lfloor j/2\rfloor}\sum_{|P|+|Q|=j-2r}\frac{1}{P!Q!}\frac{\pa^{|P|+|Q|}\al_r}{\pa z^P\pa\bar z^Q}(0)\,u^P\bar v^Q\notag\\&&+\sum_{|P|+|Q|=j+2}\frac{1}{P!Q!}\phi_{P\overline Q}(0)\left[u^P\bar v^Q-\frac12 u^P\bar u^Q-\frac12 v^P\bar v^Q\right],\label{beta}\end{eqnarray}
 where the $\alpha_r$ are given by  \eqref{alpha}.
\end{lem}

\begin{rem} The estimate \eqref{D5k} is not sharp.  A sharp estimate is given by the following result of independent interest. \end{rem}

\begin{theorem}\label {DKlog} There exist positive constants $C_{ka}$ for  $k\in\Z^+, a\in\R^+$,   such that under the hypotheses of Lemma \ref{taylor},
$$ |D^{k}\log \Pi_N(u,v)| \le C_{ka}N\,,\quad \mbox{for }\ |u|+|v|<a/\sqrtn\,,\ k\in\Z^+\,.$$
 Furthermore, 
$$ |D^{1}\log \Pi_N(u,v)| + |D^{3}\log \Pi_N(u,v)| \le C_{a}\sqrtn\,,\quad\mbox{for }\ |u|+|v|<a/\sqrtn\,.$$
\end{theorem} 

 \begin{proof}  Let $\beta_j$ be as in \eqref{logPi}. We have by \eqref{logPi} and Lemma \ref{lem25}, $$\log \Pi_N(u,v) = m\log  \frac N\pi + N(u\cdot \bar v - \half |u|^2-\half |v|^2 )+N\inv \wt E_{N2}(\sqrtn\,u\,,\sqrtn\,v)\,.$$ Differentiating the above equation up to $3$ times, we obtain the result for $k\le 3$.  
  For $k\geq 4$,   by Lemma \ref{algorithm}, $$\textstyle\deg\left[u\cdot \bar v - \half |u|^2-\half |v|^2 +\sum_{j=1}^{k-3} N^{-j/2} \beta_j  (u,v)\right] \le k-1.$$ As in the proof of Lemma \ref{taylor}, we apply \eqref{logPi} (with $k$ replaced by $k-2$) to conclude that $$|D^k\log\Pi_N(u,v)|=N^{-\frac{k-2}2}\big|D^k\big[ \wt E_{N[k-2]}(\sqrtn\,u,\sqrtn\,v)\big]\big|\le C_{ka} N\,,\quad \mbox{for }\ |u|+|v|<a/\sqrtn\,. 
$$ 
\end{proof}

 \medskip\noindent{\it Example 1:\/} 
\black Let $L$ be the hyperplane section bundle over $\CP^m$ with the Fubini-Study metric.  Let $z_j=Z_j/Z_0$ for $1\le j\le m$, where $(Z_0:\dots:Z_m)$ are the homogeneous coordinates in $\CP^m$; then $z_1,\dots,z_m$ are  K-coordinates   at $z_0=(1:0:\dots:0)$.  In terms of these K-coordinates and a K-frame at $z_0$, we have \begin{equation} \label{PiNexample} \log\Pi_N(u,{v}) =  \log \frac {(N+m)!}{\pi^mN!}\  +N\log(1+u\cdot\bar v) 
 - \frac N2\,\log(1+|u|^2)  - \frac N2\,\log(1+|v|^2)\,.\end{equation}  (See for example \cite{BSZ2}.) Differentiating, we see that $$\sup_{|u| { +|v|}<a/\sqrtn}|D^{2j}\log\Pi_N(u,{v})|\sim {c_{2j}}\, N\,,\quad \sup_{|u|{+|v|}<a/\sqrtn}|D^{2j-1}\log\Pi_N(u,{v})|\sim { c_{2j-1}}\,\sqrtn\,,$$  where the $c_k$ are constants depending on $a$.  This example   shows that Proposition \ref{DKlog} is sharp for $k$ even and for $k=1,3$.

Replacing $(u,v)$ with $(\uN,\vN)$ in \eqref{PiNexample} and exponentiating, we obtain the expansion
\begin{equation}\label{b2example} \textstyle \Pi_N(\uN,\vN) = \Pi_{BF}(u,0;v,0)\left[1+ N\inv\left( \frac{m(m+1)}2-\frac{(u\cdot\bar v)^2}2+\frac{|u|^4+|v|^4}4\right) +\cdots\right].\end{equation}
The expansion \eqref{b2example} can be used to check  Theorem \ref{main2} for the case where $M=\CP^m$.  To do this,  the reader can substitute in  Theorem \ref{main2} the values $$R_{j\bar jj\bar j}=2\,,\quad R_{j\bar jk\bar k}=R_{j\bar kk\bar j}=1 \ (j\neq k) \,,\quad \mbox{others}\, = 0\,,$$ of the curvature tensor at $z_0$, and verify that the resulting expansion coincides with \eqref{b2example}.

\bigskip\noindent {\it Example 2:\/}  To provide an example where Theorem \ref{DKlog} 
 is sharp for all $k\ge 4$, we give $M$   a \kahler metric with  potential $$\phi=|z|^2+ \sum_{j=2}^\infty(z_1^2\bar z_1^{j}+ z_1^{j}\bar z_1^2)$$ in a neighborhood of $z_0=0$.  By Lemma \ref{algorithm},
$$\be_k(u,0) = -\frac 1{2}\,(u_1^2\bar u_1^k+u_1^k\bar u_1^2) \ + \ \mbox{ terms of lower degree.}$$
 Furthermore $\beta_2,\dots, \beta_{k-1}$ are of degree $\le k+1$.
 Therefore by \eqref{logPi},
$$\frac {\d^{k+2}}{\d u_1^2\d\bar u_1^{k}}\,\log \Pi_N(\uN,0)= -\frac 12 N^{-k/2} +O(N^{-(k+1)/2})\,.$$   It  follows  that $|D^{k+2}\log \Pi_N(u,0)|\sim N$ for $k\ge 3$ (and also for $k=2$).

\bigskip
In order to use Lemma \ref{algorithm} to compute the $\beta_j$'s (and then the $b_j$'s), we need to express the derivatives of the \kahler potential $\phi$ in terms of the covariant derivatives of the curvature. Since we are using  K-coordinates with a K-frame,  each monomial in the 
 expansion \eqref{bochner} has at least two $z_j$'s and two $\bar z_j$'s.  Therefore, to compute $\beta_2,\beta_3,\beta_4$, we need  the  derivatives of $\phi$ in the following lemma.

\begin{lem}\label{lem34}
Using K-coordinates about $z_0\in M$, we have the following identities at the point $z_0$: 
\begin{align*}
&{\rm i)}\quad \phi_{ik\bar j\bar l}\ =-R_{i\bar jk\bar l}\\
&{\rm ii)} \quad \phi_{iks\bar j\bar l}\ =-R_{i\bar jk\bar l,s}\\
&{\rm iii)}\quad \phi_{ik\bar j\bar l\bar t}\ =-R_{i\bar jk\bar l,\bar t}\\
&  {\rm iv)}\quad 
\phi_{i ks\bar j\bar l\bar t}
=-R_{{i\bar jk\bar l},s\bar t}+R_{p\bar jk\bar l} R_{i\bar p s\bar t}+R_{i\bar jp\bar l} R_{k\bar ps\bar t}+R_{i\bar pk\bar t} R_{p\bar js\bar l}\\
 & {\rm v)}\quad \phi_{i kst\bar j\bar l}
\ =-R_{{i\bar jk\bar l},st}\\
& {\rm vi)}\quad \phi_{i k\bar j\bar l\bar s\bar t}
\ =-R_{{i\bar jk\bar l},\bar s\bar t}\end{align*}
Furthermore, {\rm (i)--(iii)}  hold for K-coordinates of order 3 (i.e., normal coordinates), and {\rm (iv)--(vi)} hold for K-coordinates of order 4.\end{lem}

\begin{proof}  
Equation (i) follows from \eqref{Rijkl}.
Differentiating  \eqref{Rijkl}, we have
\begin{eqnarray}
\phi_{iks\bar j\bar l}
&=& -\frac{\pa R_{i\bar j k\bar l}}{\pa z_s}+\frac{\d}{\d z_s}\left(g^{p\bar q} \,\frac{\d g_{i\bar q}}{\d z_k}\,\frac{\d g_{p\bar j}}{\d\bar  z_l}\right)\notag\\
&=& -{ R_{i\bar j k\bar l,s}}-\Gamma_{is}^p R_{p\bar jk\bar l}-\Gamma_{ks}^p R_{i\bar jp\bar l}+\frac{\d}{\d z_s}\left(g^{p\bar q} \,\frac{\d g_{i\bar q}}{\d z_k}\,\frac{\d g_{p\bar j}}{\d\bar  z_l}\right),\label{d5phi}
\end{eqnarray}
where $\Gamma_{jk}^i$ are the Christoffel symbols,
\[
\Gamma_{jk}^i=g^{i\bar l}\frac{\pa g_{j\bar l}}{\pa z_k}.
\] Thus equation (ii) follows by evaluating \eqref{d5phi} at $z_0$.   Equation (iii) is the conjugate of (ii).

We recall that $$\frac{\d g_{i\bar q}}{\d z_k}(z_0)=\frac{\d^2g_{i\bar q}}{\d z_s\d z_k}(z_0)=0$$ for   K-coordinates  of order 4. Equation (v) then follows by applying $\d/\d  z_ t$  to \eqref{d5phi}  and evaluating at $z_0$; equation (vi)  is the conjugate of (v). 

 Finally, by applying  $\d/\d \bar z_ t$ to \eqref{d5phi}, we have
$$
\phi_{i ks\bar j\bar l\bar t}(z_0)
=\left[ -{ R_{i\bar j k\bar l,s\bar t}}-\phi_{is\bar p\bar t}\,R_{p\bar jk\bar l}-\phi_{ks\bar p\bar t}\, R_{i\bar jp\bar l}+\phi_{ik\bar p\bar t} \,\phi_{ps\bar j\bar l}\right]_{z=z_0}
$$
 and using \eqref{Rijkl}  again, we get (iv).\end{proof}

\begin{proof}[Proof of Theorem~\ref{mainresult}]
By ~\cite{lu-1}*{Theorem 1.1},\footnote{The K\"ahler form we use is $\pi^{-1}$ times the  K\"ahler form in \cite{lu-1}. Therefore the expansion \eqref{Lu} differs from the one in~\cite{lu-1} by a factor $\pi^{-m}$.} 
\begin{equation}\label{Lu} \Pi_N(z,z)=\frac {N^m}{\pi^m}\left[1 +\frac 12\rho(z)\,N\inv + \left.\left(\frac13\Delta\rho+\frac 1{24}|R|^2-\frac 16|{\rm Ric}|^2 +\frac 18 \rho^2\right)\right|_z N^{-2} +O(N^{-3})\right]. \end{equation} Therefore,
\begin{equation}\label{alpha12} \al_1=\frac 12\rho\,,\quad \alpha_2=\frac 13\Delta\rho+\frac{1}{24}|R|^2-\frac16|{\rm Ric}|^2\,.\end{equation}

By Lemma~\ref{lem25}, $\beta_1=0$.  To compute $\be_2,\be_3,\be_4$, we write $$\xi_{P\overline Q} = u^Pv^{\overline Q}\,,\quad  \xi_{P\overline Q}^\sharp = u^Pv^{\overline Q}
-\half u^Pu^{\overline Q}-\half v^Pv^{\overline Q}\,.$$
Expanding Lemma~\ref{algorithm}, we have
\begin{eqnarray*} \be_2 &=& \al_1+\frac14\,\phi_{ik\bar j\bar l}\,\xi_{ik\bar j\bar l}^\sharp\;,\\
\be_3 &=& \frac{\d\al_1}{\d z_j}\,u_j + \frac{\d\al_1}{\d\bar z_j}\,\bar v_j + \frac 1{12}
\phi_{ikp\bar j\bar l}\,\xi_{ikp\bar j\bar l}^\sharp +\frac 1{12} \phi_{ik\bar p\bar j\bar l}\,\xi_{ik\bar p\bar j\bar l}^\sharp\;,\\
\be_4 &=& \al_2+ \frac{\d^2\al_1}{\d z_j\d\bar z_k}\,u_j\bar v_k +\frac 12\,\frac{\d^2\al_1}{\d z_j\d z_k}\,u_j v_k +\frac 12\,\frac{\d^2\al_1}{\d\bar z_j\d\bar z_k}\,\bar u_j\bar v_k\\
&& \ + \frac 1{48}\,\phi_{ikpq\bar j\bar l}\,\xi_{ikpq\bar j\bar l}^\sharp +\frac 1{48}\,\phi_{ik\bar p\bar q\bar j\bar l}\,\xi_{ik\bar p\bar q\bar j\bar l}^\sharp +\frac 1{36} \phi_{ikp\bar q\bar j\bar l}\,\xi_{ikp\bar q\bar j\bar l}^\sharp\;.
\end{eqnarray*}
The above formula for $\beta_j$ holds for K-coordinates with K-frame of order $j+2$ ($j=2,3,4$).

Applying Lemma~\ref{lem34}, we then obtain  
\begin{align*}
& \beta_2=\frac 12\rho+\frac 14S^\sharp(\bu,\bv)\;,\\
& \beta_3=\frac 12\nabla\rho(\bu+\bar\bv)+\frac 1{12}L^\sharp(\bu,\bv)\;,\\
&\beta_4=\frac 13\Delta\rho+\frac{1}{24}|R|^2-\frac16|{\rm Ric}|^2+\frac{1}{4}\nabla^2\rho(\bu+\bar \bv,\bu+\bar\bv)+\frac {1}{36} K_1^\sharp(\bu,\bv)+\frac {1}{48}K_2^\sharp(\bu,\bv)\;.
\end{align*} 
 Substituting these values of $\be_j$ in \eqref{logPi} and exponentiating,  we obtain the formulas of Theorem \ref{mainresult}.
\end{proof}

\begin{rem} The coefficients $b_5,b_6$ can also  be computed from Lemma \ref{algorithm} using the formula for $\al_3$ obtained from \cite{lu-1}.\end{rem}

\section{Homogeneity of the coefficients}
We now prove Theorem \ref{homogeneity}.   Recall that $\mathcal F$ 
is  the ring of polynomials of the covariant derivatives of the curvature  with coefficients in $\Q$, 
and  a polynomial is  $w$-homogeneous  if all of its monomials have the same  $w$-weight, as  given by \eqref{weight}.    For a
$w$-homogeneous polynomial, its $w$-weight  is defined as the $w$-weight of any of its monomials. 

\begin{lem}\label{3} 
Using coordinates $z_1,\dots,z_m$ that are holomorphic to infinite order at $z_0$,  the expressions
\begin{equation}\label{phi-5}
\frac{\pa^{P+Q}\phi}{\pa z^P\pa \bar z^Q}\;,
\end{equation}
 evaluated at $z_0$, are polynomials in
\begin{equation}\label{ring}
g_{i\bar j}, g^{i\bar j}, R_{i\bar jk\bar l,A\bar B}, \frac{\pa^{|C|}g_{i\bar j}}{\pa\bar z^C}, \frac{\pa^{|D|}g_{i\bar j}}{\pa z^D},
\end{equation}
where $|P|,|Q|\geq 2$, $1\leq|C|\leq|Q|-2, 2\leq|D|\leq|P|-2$.  Moreover,  if $z_1,\dots,z_m$ are K-coordinates  of order $|P|+|Q|$ at the point $z_0$, then the expression   \eqref{phi-5}, evaluated at $z_0$,     is a $w$-homogeneous polynomial in $\fcal$ of   $w$-weight  $(|P|+|Q|)/2-1$.
\end{lem}

\begin{proof} We first observe that we can replace $\frac{\pa^2\phi}{\pa z_i\pa\bar z_j}$ by $g_{i\bar j}$. 
Let $\mathcal R$ be the  ring  generated by the elements in~\eqref{ring}. 

We  need to prove that $\mathcal R$ is closed under partial derivatives and at the same time, the $w$-homogeneity is kept.
We use mathematical induction on $n=|P|+|Q|$. If $n=4$, then the conclusion is obvious because $|P|=|Q|=2$ and hence  the expression is just $-R_{i\bar jk\bar l}$. Assume that the lemma is proved for any $|P|+|Q|\leq n$. 
Using~\eqref{Rijkl}, any derivative  (evaluated at $z_0$) of the form
\begin{equation}\label{a-1}
\frac{\pa^{|C|+|D|}g_{i\bar j}}{\pa z^D\pa\bar z^C}
\end{equation}
with $|C|+|D|=n+1$, can be represented as
\begin{equation}\label{a-2}
-\frac{\pa^{n-1} R_{i\bar jk\bar l}}{\pa z^{D'}\pa\bar z^{C'}}+\frac{\pa^{n-1} }{\pa z^{D'}\pa\bar z^{C'}}\left(g^{p\bar q}\cdot\frac{\pa g_{i\bar q}}{\pa z_k}\cdot\frac{\pa g_{p\bar j}}{\pa\bar z_l}\right)
\end{equation}
for some $C',D'$ such that $|C'|+|D'|=n-1$.
 On the other hand, any derivative of the form  $\frac{\pa}{\pa z_s} R_{i\bar jk\bar l,A\bar B}$ can be represented  as a polynomial of the covariant derivatives of the curvature and the Christoffel symbols $\Gamma_{ij}^k$. Since \[
\Gamma_{ij}^k=g^{k\bar q}\cdot\frac{\pa g_{i\bar q}}{\pa z_j}\in\rcal\,,
\]
 by the inductive assumption, \eqref{a-2} is in $\mathcal R$, and the first part of the lemma is proved.

To prove the second part of the lemma,  we  extend the weight $w$ so that
\[
w(g_{i\bar j})=w(g^{i\bar j})=0, \quad w(\frac{\pa^{|P|+|Q|}g_{i\bar j}}{\pa z^P\pa \bar z^Q})=(|P|+|Q|)/2\,,
\]
and  we extend the definition of $w$ to monomials by requiring $w(AB)=w(A)+w(B)$. 
 (We remark that the $w$-weight is well-defined and is independent to the choice of local coordinates.) 
Moreover,  
such an extension is compatible  with formula~\eqref{Rijkl}  since each monomial in \eqref{Rijkl} is of weight $1$.  As a result, in the procedure of writing ~\eqref{phi-5} in terms of the covariant derivatives of the curvature, at each step, we get a $w$-homogeneous polynomial in $\mathcal R$. 
Since 
\[
\frac{\pa^{P+Q}\phi}{\pa z^P\pa \bar z^Q}=\frac{\pa^{|P'|+|Q'|}g_{i\bar j}}{\pa z^{P'}\pa\bar  z^{Q'}}
\]
for some  multiple indices $P',Q'$ such that $|P'|=|P|-1, |Q'|=|Q|-1$, 
this  degree is equal to 
\[
w(\frac{\pa^{|P'|+|Q'|}g_{i\bar j}}{\pa z^{P'} \pa \bar z^{Q'}})=\frac{|P'|+|Q'|}{2}=\frac{|P|+|Q|}2-1.
\]
At $z_0$, $g_{i\bar j}=\delta_{ij}$ and 
\[
\frac{\pa^{|C|}g_{i\bar j}}{\pa\bar z^C}=\frac{\pa^{|D|}g_{i\bar j}}{\pa z^D}=0
\]
for any $C,D\neq 0$. Thus the expression at $z_0$ is a $w$-homogeneous polynomial of the covariant derivatives of the curvature of $w$-weight 
$(|P|+|Q|)/2-1$.

\end{proof}

By the same argument, we have

\begin{lem}\label{lem36}
Let $F$ be a  $w$-homogeneous polynomial of weight $d$ in the  covariant derivatives of the curvature. Then  using K-coordinates of order $n$  at $z_0$, the expression
\[
\frac{\pa^{|P|+|Q|} F}{\pa z^P\pa\bar z^Q}
\]
is a $w$-homogeneous  polynomial of   $w$-weight   $d+(|P|+|Q|)/2$, for $|P|+|Q|\le n-2d-2$.
\end{lem}

\begin{proof}[Proof of Theorem \ref{homogeneity}]  Applying Lemmas~\ref{3} and~\ref{lem36} to formula \eqref{beta}, we see that $\beta_r$ is a polynomial in $R_{i\bar jk\bar l,P\overline Q}$, $u,\bar u,v, \bar v$, for $2\le r \le n-2$. To prove the homogeneity,  we further extend the $w$-weight by requiring  $w(f N^{-r/2})=w(f)-r/2$,  where $f\in \mathcal F$ is a $w$-homogeneous polynomial 
and $r\in\mathbb Z$.  
A formal series is called {\it regular}, if all of its monomials are of weight  $0$. 
 The set of regular formal series forms an algebra over $\mathbb C$. 
Then by~\cite{lu-1}*{Theorem 1.1},  the Catlin-Zelditch asymptotic series \eqref{CZ} is regular.
It follows immediately that the asymptotic expansion ~\eqref{alpha} is regular,  and therefore by 
Lemma~\ref{lem36},
 $${ \sum_{j=1}^s } N^{-\frac{|P|+|Q|}{2}-j}\,\frac{\pa^{|P|+|Q|}\alpha_j}{\pa z^P\pa\bar z^Q}{ (z_0)}$$ is regular for $s+|P|+|Q|\le n-2$.  
 By 
 Lemma~\ref{3},   $$\sum_{|P|+|Q|\le n}  N^{-\frac{|P|+|Q|}{2}+1}\,\phi_{P\overline Q}(z_0)$$ is also regular. Thus by Lemma~\ref{algorithm}, the partial sum from \eqref{logPi},
 $$\Sigma_r:=\sum_{j=2}^r N^{-j/2}\be_j$$ is regular.  Therefore 
 \begin{equation}\label{expSigma} \exp(\Sigma_r)= 1+N\inv b_2+\cdots+ N^{-r/2}b_r+O(N^{-\frac{r+1}2})\end{equation} is a regular series,   and hence $w(b_r)=r/2$.

To prove the second part of the theorem,  we  recall from Lemma \ref{algorithm} that  $\beta_r$ is a polynomial of degree at most $r+2$ in $(u,v,\bar u,\bar v)$.
It then follows from \eqref{expSigma} that the polynomial $b_r$ is of the form
\begin{equation}\label{bb} b_r=\sum \left\{q_{a_1\cdots a_t}\prod_{i=1}^t\beta_{a_i}\ :\ 1\le t\le\left\lfloor \frac r2\right\rfloor,\; 2\le a_1\le\cdots\le a_t,\; \sum_{i=1}^t a_i=r\right\}\,,\end{equation}  where $q_{a_1\cdots a_t}\in\Q$. Since $\deg \beta_{a_i} \le  a_i+2$, it follows that
\begin{equation}\deg b_r\le   \;\max_{a_1,\dots,a_t}\,  \sum_{i=1}^t (a_i+2) =   r+2\left\lfloor \frac r2\right\rfloor  = \left\{ \begin{array}{ll} 2r\,,&\ \mbox{for $r$ even,}\\2r-1\,, &\ \mbox{for $r$ odd.}\end{array}\right.\label{bbr}\end{equation}  
 Indeed, for $r=2s$, the sum in \eqref{bb} contains the term $\frac{1}{s!}(\beta_2)^s$, which contributes\break $\frac 1{s!4^s} S^\sharp(\bu,\bv)^s$, which gives the terms of top degree $2r$ in $b_r$. Similarly, if $r=2s+1$, then $b_r$ contains ${ \frac 1{3\cdot 4^s(s-1)!}} S^\sharp(\bu,\bv)^{s-1}L^\sharp(\bu,\bv)$, which is homogeneous of degree $4s+1=2r-1$. Thus we have equality in \eqref{bbr}.
\end{proof}

\begin{rem} It follows  from \eqref{beta} that   $\beta_r$ has the same parity as $r$.   Thus each monomial of $\be_{a_i}$ has degree equal to $a_i\!\! \mod 2$,  and \eqref{bb} then gives an alternative proof that $b_r$ also has the same parity as $r$.
\end{rem}

\end{document}